\documentclass{amsart}
\usepackage{amsthm,amsfonts,graphicx, verbatim}

\theoremstyle{definition}
\newtheorem{definition}{Definition}
\newtheorem{proposition}{Proposition}

\newtheorem{corollary}{Corollary}

\newtheorem{theorem}{Theorem}
\newtheorem{example}{Example}

\newcommand{\dontshow}[1]{}
\begin{document}

\title{Betti numbers of Toric Varieties and Eulerian Polynomials}

\author{Letitia Golubitsky}

\address{\newline
\hphantom{xxi}Department of Mathematics\newline
\hphantom{xxi}University of Western Ontario\newline
\hphantom{xxi}Middlesex College\newline
\hphantom{xxi}London, Ontario\newline
\hphantom{xxi}N6A 5B7,Canada 
}

\email{lbanu@uwo.ca}

\maketitle

\begin{abstract}
It is well-known that the Eulerian polynomials, which count permutations in $S_n$ by their number of descents, give the $h$-polynomial/$h$-vector of the simple
polytopes known as permutohedra, the convex hull of the $S_n$-orbit for a generic weight in the weight lattice of $S_n$. Therefore the Eulerian polynomials give the Betti numbers for certain smooth toric varieties associated with the
permutohedra.

In this paper we derive recurrences for the $h$-vectors of a family of polytopes generalizing this. The simple polytopes we consider arise as the orbit of a non-generic weight, namely a weight fixed by only the simple reflections $J=\{s_{n},s_{n-1},s_{n-2} \cdots,s_{n-k+2},s_{n-k+1}\}$ for some $k$ with respect to the $A_n$ root lattice. Furthermore, they give rise to certain rationally smooth toric varieties $X(J)$ that come naturally from the theory of algebraic monoids. Using effectively the theory of reductive algebraic monoids and the combinatorics of simple polytopes, we obtain a recurrence formula for the Poincar\'e polynomial of $X(J)$ in terms of the Eulerian polynomials.
\end{abstract}

\section{Introduction.}

Toric varieties and their cohomology have played an increasingly important role in studying the combinatorics of convex polytopes. They started around 1980 with Stanley's spectacular proof of the necessity of McMullen's conditions (characterizing the face numbers of a simple polytope) using the cohomology of rationally smooth projective toric varieties. This connection between the topology of toric varieties and the combinatorial geometry of convex polytopes is of interest to us.
\let\thefootnote\relax\footnotetext{{\it Date:} September 4, 2010}

Let $(W,S)$ be a finite Weyl group of type $A_n$. Let $J$ be any proper subset of $S$. Associated with $J$ is a certain projective toric variety $X(J)$. We would like to calculate the Betti numbers of $X(J)$ when $J$ is combinatorially smooth, {\it i.e.}, $X(J)$ is a rationally smooth variety.

\begin{definition} \cite{Br3}
Let $X$ be a complex algebraic variety of dimension $n$. Then $X$ is $\emph{rationally smooth at $x$}$ if there is a neighbourhood $U$ of $x$ in the complex topology such that, for any $y \in U$, $$H^m(X, X \setminus \{y\})=0, \ m \ne 2n$$ $$H^{2n}(X, X \setminus \{y\})=\mathbb{Q}.$$
Here $H^*$ denotes the cohomology of $X$ with rational coefficients.
\end{definition}

The most basic combinatorial data of a $d$-dimensional convex polytope are the numbers $f_i$ of $i$-dimensional faces encoded in the face polynomial $f(t):=\sum_{i=0}^d f_i t^i$. For simple polytopes, {\it i.e.}, where each vertex lies on exactly $d$ edges, the possible $f$-polynomials are expressed in terms of the $h$-polynomials $h(t)=f(t-1)=\sum_{i=0}^d h_i t^i$ where $h_i$ are strictly positive and satisfy the symmetry relation $h_i=h_{d-i}$. When a polytope $P$ is rational, {\it i.e.}, all its vertices have rational coordinates with respect to some lattice, we associate a toric variety to it $X_P$ using the normal fan construction. It turns out that the Poincar\'e polynomial of $X_P$ agrees with $h(t^2)$. Danilov \cite{D} proved that $X_P$ is rationally smooth if and only if the polytope $P$ is simple.

Consider a semisimple algebraic group $G_0$ with maximal torus $T_0$ and an irreducible representation $\rho_{\lambda}$ of $G_0$ with the highest weight $\lambda \in X(T_0) \otimes \mathbb{Q}$. Consider the action of $W$ on the vector space spanned by the simple roots of $G_0$ and take the convex hull of the $W$-orbit of $\lambda$, $P_{\lambda}={\rm Conv}(W.\lambda) \subset X(T_0) \otimes \mathbb Q$. Using the inner normal fan construction associated to the polytope $P_{\lambda}$ \cite{WF}, we obtain a projective toric variety $X(J)$. The terminology is justified since $X(J)$ depends only on $$J=\{s \in S  \ | \ s(\lambda)=\lambda\}.$$
In \cite{R6} Renner finds necessary and sufficient conditions for the polytope $P_{\lambda}$ to be simple using the theory of algebraic monoids  that he developed along with Putcha since 1980. For each Weyl group $(W,S)$, Renner gives a classification of all $J \subseteq S$, such that $X(J)$ is rationally smooth. See Corollary 3.5 in \cite{R6}. In this paper we are interested only in the case of $(W,S)$ finite Weyl group of type $A_n$. 
\begin{definition} \cite{R4}
We refer to $J$ $\emph{combinatorially smooth}$ if $P_{\lambda}$ is a simple polytope.\\

An equivalent definition corresponds to $J$ is $\emph{combinatorially smooth}$ if the variety $X(J)$ is rationally smooth.
\end{definition}

 According to Renner's classification, the subset $J$ is combinatorially smooth of type $A_n$ if $J \subset \{s_1,s_2, \cdots,s_n\}$ has one of the following forms:
\begin{enumerate}
\item $J_0=\emptyset$,
\item $J_0=\{s_1, \cdots ,s_i\}$ where $1 \le i \le n$,
\item $J_0=\{s_j, \cdots ,s_n\}$ where $1<j \le n$,
\item $J_0=\{s_1, \cdots ,s_i, s_j, \cdots ,s_n\}$ where $1 \le i,j \le n$ and $j-i \le 3$.
\end{enumerate}

When $(W,S)$ is finite Weyl group of type $A_n$ and $J=\emptyset$, the polytope $P_{\lambda}$ is a permutahedron. The Betti numbers of $X(\emptyset)$ are given by the Eulerian numbers. In \cite{Br}, Brenti studies the descent polynomials ({\it i.e.}, the Poincar\'e polynomials of $X(\emptyset)$) as analogues of the Eulerian polynomials.

When $J \ne \emptyset$, the weight $\lambda$ is allowed to lie on certain reflecting hyperplanes. Of course, the orbit of a point in the complement of the arrangement is just the ordinary permutahedron. In the case of $J \ne \emptyset$ combinatorially smooth subset of $S$, whether the Poincare polynomial of $X(J)$ can be expressed in terms of the Eulerian numbers is an interesting question. In this article we answer this question by computing the Poincar\'e polynomial of $X(J)$ when $(W,S)$ is finite Weyl group of type $A_n$ and $J=\{s_{n-k+1}, \cdots ,s_n\} \subseteq S$ is combinatorially smooth, with $s_k=(k, \ k+1) \in S_n$ and $1 \le k \le n$. One needs to investigate further to see whether our new technique can provide answer for all types of combinatorially smooth sets $J$.

Our first result deals with the case when the highest weight $\lambda$ is fixed only by the reflection $s_n =(n,n+1) \in S_n$. We obtain the following characterization of the $h$-polynomial of $X(J)$ in terms of the Eulerian polynomials.
\vspace{5mm}

{\bf Theorem:} Let $J=\{s_n\} \subset S$. Then $J$ is combinatorially smooth and the $h$-polynomial of $X(J)$ is given by
$$
h(t)=E_{n+1}(t)- \binom{n+1}{2} t E_{n-1}(t).
$$

Then we generalize the computations to the case of $J=\{s_{n-k+1}, \cdots ,s_n \} \subseteq S$ for $1 \le k \le n$. Our main result is a simple recursive relation for the Poincar\'e polynomial of $X(J)$ in terms of the $(n-k)$-Eulerian polynomials.

\vspace{5mm}

{\bf Theorem:} Let $J(k,n)=\{s_{n-k+1}, s_{n-k+2}, \cdots ,s_n\} \subset S$, $1 \le k \le n$ and let $h_k(t)$ denote the $h$-polynomial of $X(J(k,n))$. Then $J(k,n)$ is combinatorially smooth and the following recurrence relation holds:
$$
h_k(t)= h_{k-1}(t)- \binom{n+1}{k+1} (t^{k} +t^{k-1}+ \cdots +t)E_{n-k}(t).
$$

Finally, the recurrence relation is illustrated for $J(n-1,n)=\{s_2, s_3, s_4, \cdots ,s_n\}$ and $J(n-2,n)=\{s_3, s_4, \cdots ,s_n\}$ where the $h$-polynomial of $X(J(n-2,n))$ and $X(J(n-1,n))$ are computed in \cite{R6} and \cite{R7}.

This paper is structured as follows. In Section 2 we introduce briefly $\mathcal{J}$-irreducible monoids of type $J$ and the projective toric variety associated $X(J)$. Using the notion of cross section lattice associated to a reductive monoid we have a formula presented in Proposition 1, for the number of $i$-dimensional faces of the polytope $P_{\lambda}$ corresponding to the variety $X(J)$. This gives us a good handle of the $h$-polynomial of $X(J)$, which can be expressed in terms of all subsets of $S$. There is an interesting interplay between the geometry of $X(J)$ and the combinatorics of finite sets, illustrated in Corollary 1. In Section 3 we introduce Eulerian polynomials and prove our main results. We conclude this section with an example that illustrates the recurrence formula obtained in Theorem 6.

\section{$h$-polynomial of $X(J)$}
\vspace{5mm}

Throughout the paper we work with the field $\mathbb{C}$ of complex numbers. In this section we establish the definitions and the results needed throughout this text. A very good up-to-date account of the theory of algebraic monoids can be found in \cite{R4} and \cite{Sol}.

A linear algebraic monoid is an affine variety together with an associative morphism and an identity element. An irreducible monoid $M$ is called \emph{reductive} if its unit group $G$ is a reductive group. Let $B$ be a Borel subgroup of $G$ and $T \subset B$ a maximal torus of $G$. The set of idempotents in $\overline{T}$ is defined as $$E(\overline{T})=\{e \in \overline{T} \ | \ e^2=e\},$$where $\overline{T}$ is the Zariski closure of $T$ in $M$. The set $$\Lambda=\{e \in E(\overline{T}) \ | \ Be=eBe\}$$ is called the cross section lattice of $M$ relative to $B$ and $T$.

\begin{definition} \cite{R4}
A reductive monoid $M$ with $ 0 \in M$ is called $\mathcal{J}$-irreducible if $M -\{0\}$ has exactly one minimal $G \times G$-orbit. See section 7.3 of \cite{R4} for a systematic discussion of the important class of reductive monoids and for a proof of the following theorem.
\end{definition}

\begin{definition}
A morphism $f: X \to Y$ is a finite morphism of algebraic varieties if \mbox{$f^*:K[Y] \to K[X]$} makes $K[X]$ into a finitely generated module over $k[Y]$.
\end{definition}

\begin{theorem} \cite{R4} \label{theo: monoid} Let $M$ be a reductive monoid. The following are equivalent:

\begin{enumerate}
\item $M$ is $\mathcal{J}$-irreducible.
\item There is an irreducible rational representation $\rho:M \to End(V)$ which is finite as a morphism of algebraic varieties.
\item If $\overline{T} \subset M$ is the Zariski closure in $M$ of a maximal torus $T \subset G$ then the Weyl group $W$ of $T$ acts transitively on the set of minimal nonzero idempotents of $\overline{T}$.
\end{enumerate}
\end{theorem}
 
\begin{definition} \cite{R6}
If $M$ is $\mathcal{J}$-irreducible, we say that $M$ is $\mathcal{J}$-irreducible of type $J$ if $$J=\{s \in S \ | \ s e_1=e_1 s\},$$ where $S$ is the set of reflections relative to $T$ and $B$ and $e_1$ is the minimal idempotent such that $e_1B=e_1Be_1$.
\end{definition}

Next, we describe the $G \times G$-orbit structure of a $\mathcal{J}$-irreducible monoid of type $J \subseteq S$. 
First, recall the partial ordering on the $G \times G$-orbits described as follows: 
$$
GeG \prec GfG \ {\rm if  \ and \  only \ if \ } GeG \subset \overline{GfG} \ {\rm if \ and \ only \ if} \ ef=e
$$
The following result was first presented in Theorem 4.16 \cite{PR}.

\begin{theorem} \cite{PR} \label{theo: lattice} Let $M$ be a $\mathcal{J}$- irreducible monoid of type $J \subseteq S$.
\begin{enumerate}
\item There is a canonical one-to-one order-preserving correspondence between the set of $G \times G$- orbits acting on $M$ and the set of $W$-orbits acting on the set of idempotents of $\overline{T}$. This set is canonically identified with $\Lambda =\{e \in E(\overline{T}) \ | \ eB=eBe\}.$
\item $\Lambda -\{0\} \cong \{I \subseteq S \ | \ {\rm no \ connected \  component \ of\ }I{\rm \ is \ contained \ entirely \  in\ } J \}$ in such a way that $e$ corresponds to $I(e) \subset S$ if $I(e)=\{s \in S \ | \ se=es \ne e \}.$
\item If $ e \in \Lambda -\{0\}$ corresponds to $I(e)$, as in 2 above, then $C_W(e)=W_{I^{*}(e)}$, where 

$I^{*}(e)=I \cup \{s \in J \ | \ st=ts \ {\rm for \ all \ t} \ \in I(e) \}.$
\end{enumerate}
\end{theorem}

The above theorem, which will be used throughout the paper is the bridge between a cross section lattice and the combinatorics of finite sets.

\vspace{5mm}

Consider a $\mathcal{J}$-irreducible monoid $M$ of type $J$ and an irreducible representation \mbox{$\rho:M \to End_{\mathbb{C}}(V)$} which is finite as a morphism.

$M$ has a reductive unit group $G$. Let $B \subset G$ be a Borel subgroup of $G$ and $T \subset B$ a maximal torus of $G$. Let $\overline{T}$ denote the Zariski closure of $T$ in $M$. By Theorem 5.4 of \cite{R4}, $\overline{T}$ is a normal affine toric variety.

The dimension of the $T$-orbit corresponding to an idempotent $e \in \overline{T}$ is called the \emph{rank} of $e$ {\it i.e.}, ${\rm rank}(e)={\rm dim} Te$.

 Let $G_0$ be a semisimple algebraic group defined as the commutator subgroup $G_0=(G,G)$ of $G$ with maximal torus $T_0=T \ \cap \ G_0$ and let $\rho_{\lambda}=\rho_{|G_0}$ be the representation of $G_0$ that corresponds to the highest weight $\lambda \in X(T_0)$. Let the Weyl group $W$ act by reflections on the vector space spanned by the roots of $G$. Take the $W$-orbit of $\lambda$ and consider its convex hull in $X(T_0) \otimes \mathbb{Q}$. We obtain the polytope $P_{\lambda}={\rm Conv}(W. \lambda) \subset X(T_0) \otimes \mathbb{Q}$.

In the cases of interest in this paper, we consider the 1-1 correspondence, as posets, between $E(\overline{T}) \setminus \{0\}$ and the face lattice of the polytope $P_{\lambda}$, namely  \begin{equation}e \in E(\overline{T}) \leftrightarrow \mathcal{F}_e,
\end{equation} such that ${\rm rank}(e)={\rm dim}(\mathcal{F}_e)+1$.
For more details on this correspondence see \cite{R2}. 

We know that $T$ is a Zariski open subset of $\overline{T}$. Hence the torus $$\frac{T}{\mathbb{C}^*} \ {\rm is \ an \ open \ subset \ of} \ \frac{\overline{T}-\{0\}}{\mathbb{C}^*}.$$

Our interest is in the projective toric variety:
$$
X(J)=\frac{\overline{T}-\{0\}}{\mathbb{C}^*}={\rm Proj}[\mathbb{C}[\overline{T}]],
$$
that depends only on $J=\{s \in S \ | \ s(\lambda)=\lambda\}$ and not on $\lambda$ or $M$.

Next, we mention a formula for calculating the number of $i$-dimensional faces of the polytope $P_{\lambda}$ using the  lattice isomorphism (1) between $E(\overline{T})$ and the face lattice of the polytope $P_{\lambda}$. The formula can be found in Lemma 4.1 in \cite{Li}. The proof is omitted there and because it serves our purpose we include it in this paper.

\begin{proposition}\cite{Li} \label{pro: face} The number of $i$-dimensional faces of $P_{\lambda}$ is:
$$f_i=\sum_{e \in \Lambda_{i+1}} \frac{|W|}{|W_{I^*(e)}|},$$
where $\Lambda_{i+1}=\{e \in \Lambda \ | \ {\rm rank}(e)={\rm dim}(Te)=i+1\}$ and to  $I(e)=\{s \in S  \ | \ se=es\ne e\}$ it corresponds $I^*(e)= I(e) \cup \{s \in J \ | \ st=ts \ {\rm for \ all} \ t \in I(e)\}.$
\end{proposition}
\begin{proof}

Let $\mathcal{F}_i$ be the set of all $i$-dimensional faces of the polytope $P_{\lambda}$. We know that $W.\mathcal{F}_i=\mathcal{F}_i$, as the Weyl group permutes the $i$-dimensional faces of $P_{\lambda}$. The Weyl group $W$ is acting on $E(\overline{T})$ by conjugation. Then for any $e \in E(\overline{T})$, the isotropy group of $e$ is the centralizer of $e$ in $W$, namely $W_{I^{*}(e)}$, according to Theorem~\ref{theo: lattice}.

Hence we get:$$f_i=|\mathcal{F}_i|=|W.\mathcal{F}_i|=\sum_{e \in \Lambda_{i+1}} W.e=\sum_{e \in \Lambda_{i+1}} \frac{|W|}{|W_{I^{*}(e)}|}.$$
\end{proof}

This proposition yields an interesting formula of the $h$-polynomial, which will be used frequently in proving our main results.

Recall that the nodes of a Dynkin diagram corresponds to some special vectors, called $\emph{simple roots}$, in the vector space $V=X(T) \otimes \mathbb{Q}$. The set of simple roots is denoted by $\Delta$.

A subset $Y \subseteq \Delta$ is called $\emph{connected}$ if it is a connected subset of the underlying graph of the Coxeter diagram.

For $(W,S)$ finite Weyl group, the $\emph{graph structure}$ on $S$ is defined as follows:
$$
s \ {\rm and} \ t \ {\rm are \ joined \ by \ an \ edge \ if} \ st \ne ts.
$$

Let $S(J)=\{I \subset S \ | \ {\rm no \ connected \  component \ of\ }I{\rm \ is \ contained \ entirely \  in\ } J \}$. From Theorem 2 we have that $\Lambda \setminus \{0\} \cong S(J)$. Define for any $I \in S(J)$,
$$
I^*_J=I \cup \{s \in J \ | \ st=ts \ {\rm for \ all} \ t \in I\}.
$$

\begin{corollary} \label{cor: hpoly}
The $h$-polynomial of $X(J)$ is given by:
$$h(t)=\sum_{I \subseteq S(J)} \frac{|W|}{|W_{I^*_J}|} (t-1)^{|I|}.$$
\end{corollary}

\begin{proof}

According to Theorem 2, to each element $e \in \Lambda \setminus \{0\}$ corresponds uniquely a subset of $S(J)$, denoted by $I(e)$ such that $$I(e)=\{s \in S  \ | \ se=es\ne e\}$$ and ${\rm rank}(e)=|I(e)|+1$. We associate to $I(e)$ the following set $$I^*_J(e)= I(e) \cup \{s \in J \ | \ st=ts \ {\rm for \ all} \ t \in I(e)\}.$$

Under the correspondence (1) we have that ${\rm rank}(e)={\rm dim} \mathcal{F}_e+1$ where $\mathcal{F}_e$ is the face of the polytope $P_{\lambda}$ that corresponds uniquely to $e \in E(\overline{T}) \setminus \{0\}$.

We know that the $h$-polynomial of $X(J)$ is defined in terms of the $f$-polynomial, {\it i.e.}, $h(t)=\sum_{i=0}^d f_i (t-1)^i$, where $f_i$ is the number of $i$-dimensional faces of the polytope $P_{\lambda}$. Using the previous theorem and the fact that $$\Lambda=\bigsqcup_{i=0}^{d} \Lambda_{i+1},$$ we conclude that the $h$-polynomial is given by the following formula:
$$
\begin{array}{rcl}
h(t)&=&\displaystyle\sum_{i=0}^d f_i (t-1)^i=\displaystyle\sum_{i=0}^d \sum_{e \in \Lambda_{i+1}} \frac{|W|}{|W_{I^*(e)}|} (t-1)^i\vspace{3mm}\\
&=&\displaystyle\sum_{i=0}^d \sum_{e \in \Lambda_{i+1}} \frac{|W|}{|W_{I^*(e)}|} (t-1)^{{\rm rank}(e)-1}\vspace{3mm}\\
&=&\displaystyle\sum_{i=0}^d \sum_{e \in \Lambda_{i+1}} \frac{|W|}{|W_{I^*(e)}|} (t-1)^{|I(e)|}\vspace{3mm}\\
&=&\displaystyle\sum_{e \in \Lambda \setminus \{0\}} \frac{|W|}{|W_{I^*_J(e)}|}(t-1)^{|I(e)|}
\end{array}
$$
To ease the notation in the preceding formula we replace for every $e \in \Lambda \setminus \{0\}$ the corresponding set $I(e) \subset S(J)$ by $I \subset S(J)$. We know from Theorem 2 that $\Lambda\setminus \{0\} \cong S(J)$ hence, this yields the desired formula.

\end{proof}

\section{Betti numbers of $X(J)$ in terms of the Eulerian polynomials}
\vspace{5mm}

\begin{definition}
A permutahedron $P_{n-1} \in {\mathbb R^{n}}$ is the convex hull in ${\mathbb R^{n}}$ of the set
$$
\{(p_1, p_2, \cdots ,p_{n}) \in {\mathbb R^{n}} \ | \ (p_1 \ p_2 \cdots p_{n}) \in S_{n} \}.
$$
\end{definition}

Next we introduce Eulerian polynomials associated to a  Coxeter system of type $A_n$. More results on this topic can be found in \cite{FH} and \cite{BjBr}.
\vspace{5mm}

Let $\sigma=(p_1, \cdots ,p_n) \in S_n$. Define the ascent set of $\sigma$
$$
A(\sigma)=\{i \ | \ 1 \le i \le n: \ p_i <p_{i+1}\}.
$$
It turns out that $$i \in A(\sigma) \Longleftrightarrow l(\sigma s_i)=l(\sigma)+1.$$ Here the notation $l(\sigma)$ stands for the length of a permutation $\sigma \in S_n$.

Let $$E(n,i)=|\{\sigma \in S_n \ | \ |A(\sigma)=i\}|.$$ be the Eulerian numbers.
When $(W,S)$ is finite Weyl group of type $A_{n-1}$ we define the $n$-Eulerian polynomials as follows:
$$E_n(t)=\sum_{i=0}^{n-1} E(n,i) t^i.$$

\begin{theorem}\label{theo: eulerian1} Let $h_{n-1}$ be the $h$-polynomial of the permutahedron $P_{n-1}$, and let $E_n$ be the $n$-Eulerian polynomial. Then $$h_{n-1} (t)=E_n(t).$$
\end{theorem}

Renner gives a proof of this known fact \cite{Pro} in Theorem 3.1 \cite{R10} using algebraic topology and ascent polynomials. As a consequence we obtain a characterization of the $(n+1)$-Eulerian polynomial in terms of all subsets $I \subseteq S$. Recall that $S=\{s_1, \cdots ,s_n\}$ is the minimal set of reflections that generates the permutation group $S_{n+1}$ with $s_i=(i, \ i+1) \in S_{n+1}$.

\begin{theorem}\label{theo: eulerian2} Let $E_{n+1}$ be the $(n+1)$-Eulerian polynomial. The following identity holds:
$$     
E_{n+1}(t)= \sum_{I \subseteq S} \frac{(n+1)!}{|W_I|} (t-1)^{|I|}.
$$
\end{theorem}
\begin{proof}

Consider $(W,S)=(S_{n+1}, S)$ finite Weyl group of type $A_n$ and let $J=\emptyset$. When $J=\emptyset$, the highest weight $\lambda$ is in the interior of the fundamental Weyl chamber.  By applying reflections $s_i=(i, \ i+1) \in S_{n+1}$ about the hyperplanes orthogonal to the simple roots we permute $i$ and $i+1$ coordinates of $\lambda$. The polytope $P_{\lambda}$ given by the convex hull of the $W$-orbit of $\lambda$ turns out to be an $n$ permutahedron since its vertices are obtained by permuting their coordinates. From Theorem 3 and Corollary 1 we have that $$h(t)=E_{n+1}=\sum_{I \subseteq S(J)} \frac{|W|}{|W_{I^*_J}|}(t-1)^{|I|}.$$

We observe that when $J=\emptyset$ for every $I \subseteq S(J)$ the following holds: $$I^*_J=I \cup \{s \in J \ | \ st=ts \ {\rm for \ all} \ t \in I\}=I.$$

\end{proof}

For the remainder of this section we specialize the discussion to the case of finite Weyl group of type $A_n$ where $W=S_{n+1}$ and $S=\{s_1, \cdots ,s_n\}$.

\begin{theorem} \label{theo: eulerian3}Let $J=\{s_n\}$. Then $J$ is combinatorially smooth and the $h$-polynomial of $X(J)$ is given by
$$
h(t)=E_{n+1}(t)- \binom{n+1}{2} t E_{n-1}(t).
$$
\end{theorem}

\begin{proof}
 
From Corollary 3.5 \cite{R6} we obtain that $J=\{s_n\}$ is combinatorially smooth. Thus $P_{\lambda}$ is a simple polytope and so $X(J)$ is a rationally smooth variety.

Let $M$ be a $\mathcal{J}$-irreducible monoid of type $J=\{s_n\}$. We associate to $M$ the cross section lattice denoted by $\Lambda(1)$. Recall that
$$S(J(1,n))=\{I \subseteq S \ | \ {\rm no \ connected \  component \ of\ }I{\rm \ is \ contained \ entirely \  in\ } J \}.$$ From Theorem~\ref{theo: lattice} we have that $\Lambda(1) \setminus \{0\} \cong S(J(1,n))$.

Furthermore, we want to partition $S(J(1,n))$ into disjoint sets, some of which contain $s_{n-1}$ and the others do not contain $s_{n-1}$. This can be done in the following way:
$$
S(J(1,n))=\{A \subseteq S \ | \ s_{n-1} \in A \} \sqcup \{A \subseteq S \ | \ s_{n-1} \notin A, s_n \notin A \}.
$$   
Notice that for $A \subseteq S(J(1, n)) \setminus \{0\}$ such that $s_{n-1} \notin A$ we have $s_n \notin A$, since $\{s_n\}$ is a connected component of A contained entirely in J, and this contradicts the definition of $S(J(1,n))$.

We are now in position to introduce notation for the two disjoint subsets of $S(J(1,n))$.
Let $$
\begin{array}{rcl}
M_0&=&\{A \subseteq S \ | \ s_{n-1} \in A \}. \vspace{3mm}\\
M_1&=&\{A \subseteq S \ | \ s_{n-1} \notin A, s_n \notin A \}.
\end{array}
$$

Then the $h$-polynomial of $X(J)$ is computed using Corollary 1:
$$h(t)=\sum_{A \subseteq M_0} \frac{(n+1)!}{|W_{A^{*}}|} (t-1)^{|A|} +\sum_{A \subseteq M_1} \frac{(n+1)!}{|W_{A^{*}}|} (t-1)^{|A|}.  
$$
Next, we determine $$A^{*}=A \cup \{s \in J \ | \ st=ts \ {\rm for  \ any} \ t \in A \},$$ for any $A \subseteq S(J(1,n))$.

For $A \subseteq M_0$ we have $A^{*}=A  \ {\rm and} \ W_{A^{*}}=W_A$.

For $A \subseteq M_1$ we have $A^{*}=A \cup \{s_n\}$ and $W_{A^{*}}$ is the subgroup generated by $A \cup \{s_n\}$ namely, $$W_{A^{*}}=W_A \times S_2.$$ 

It follows that the $h$-polynomial of $X(J)$ is given by:
\begin{equation}
\begin{array}{rcl}
h(t)&=&\displaystyle\sum_{A \subseteq M_0} \frac{(n+1)!}{|W_{A}|} (t-1)^{|A|} +\sum_{A \subseteq M_1} \frac{(n+1)!}{|W_{A} \times S_2|} (t-1)^{|A|}\; \vspace{3mm}\\
 &=& \displaystyle\sum_{A \subseteq M_0} \frac{(n+1)!}{|W_{A}|} (t-1)^{|A|}+ \frac{n(n+1)}{2} \sum_{A \subseteq M_1} \frac{(n-1)!}{|W_A|} (t-1)^{|A|}.
\end{array}
\end{equation}

By definition $M_1=S \setminus \{s_n, s_{n-1}\}$ and using Theorem~\ref{theo: eulerian2} we are able to express the \mbox{$(n-1)$}-Eulerian polynomial in terms of the subsets of $M_1$. We obtain the following:
$$E_{n-1}(t)=\sum_{A \subseteq M_1 } \frac{(n-1)!}{|W_A|}(t-1)^{|A|}.$$
\noindent By (1), this implies that the $h$-polynomial of $X(J)$ equals: 
\begin{equation}
h(t)= \sum_{A \subseteq M_0} \frac{(n+1)!}{|W_{A}|} (t-1)^{|A|}+ \frac{n(n+1)}{2} E_{n-1}(t).
\end{equation}
Next, let $N_0=\{A \subseteq S \ | \ s_{n-1} \notin A \}$ and apply Theorem~\ref{theo: eulerian2} to the $(n+1)$- Eulerian polynomial. The following identity is obtained:
\begin{equation}
\begin{array}{rcl}
E_{n+1}(t)&=&\displaystyle\sum_{A \subseteq S} \frac{(n+1)!}{|W_A|} (t-1)^{|A|} \vspace{3mm}\\
& =&\displaystyle\sum_{A \subseteq M_0} \frac{(n+1)!}{|W_A|} (t-1)^{|A|} + \sum_{A \subseteq N_0} \frac{(n+1)!}{|W_A|} (t-1)^{|A|}.
\end{array}
\end{equation}
\noindent In the preceding formula we rearrange the set $N_0$ in terms of the subset $M_1$ and obtain: 
$$N_0=\{A \subseteq S \ | \ s_{n-1} \notin A, s_n \notin A \} \sqcup \{A \subset S \ | \ s_{n-1} \notin A, s_n \in A \}.$$
Let $$N_1=\{A \subseteq S \ | \ s_{n-1} \notin A, s_n \in A \}=\{A' \cup \{s_n\} \ | \ A' \subseteq \{s_1, s_2, \cdots ,s_{n-2}\}\},$$
such that $N_0=M_1 \sqcup N_1.$

Direct computation shows that the subgroup generated by $A$ can be expressed as follows: for $A \subseteq N_1$ there exists $A' \subseteq \{s_1,s_2, \cdots s_{n-2}\}$ such that $A=A' \cup \{s_n\}$ and $$W _A =W_{A'} \times S_2.$$
Therefore we can evaluate the following summand of $E_{n+1}(t)$:
\begin{equation}
\begin{array}{rcl}
\displaystyle\sum_{A \subseteq N_0} \frac{(n+1)!}{|W_A|} (t-1)^{|A|}&=&\displaystyle\sum_{A \subseteq M_1} \frac{(n+1)!}{|W_A|} (t-1)^{|A|}+ \sum_{A \subseteq N_1} \frac{(n+1)!}{|W_A|} (t-1)^{|A|} \vspace{3mm}\\
& =&\displaystyle\sum_{A \subseteq M_1} \frac{(n+1)!}{|W_A|} (t-1)^{|A|}+ \vspace{3mm}\\
&& \displaystyle\sum_{A' \subseteq M_1} \frac{(n+1)!}{|W_{A'}|| S_2|} (t-1)^{|A'|+1}\vspace{3mm}\\
& =&\displaystyle \sum_{A \subseteq M_1} \frac{(n+1)!}{|W_A|} (t-1)^{|A|}\;+\vspace{2mm}\\
&& \displaystyle\frac{n(n+1)}{2}(t-1)\sum_{A' \subseteq M_1} \frac{(n-1)!}{|W_{A'}|} (t-1)^{|A'|}\vspace{3mm}\\
& =&\displaystyle n(n+1)E_{n-1}(t)+\frac{n(n+1)}{2} (t-1)E_{n-1}(t) \vspace{3mm}\\& =&\displaystyle\frac{n(n+1)}{2}(t+1)E_{n-1}(t).
\end{array}
\end{equation}

From (4) and (5) we obtain the first summand of the $h$-polynomial stated in (2):
\begin{equation}
\sum_{A \subseteq M_0} \frac{(n+1)!}{|W_A|}=E_{n+1}(t)-\frac{n(n+1)}{2} (t+1) E_{n-1}(t).
\end{equation}

By (2) and (6), this implies that the following relation holds:
$$
\begin{array}{rcl}
h(t)&=&\displaystyle E_{n+1}(t)-\frac{n(n+1)}{2}(t+1)E_{n-1}(t)+\frac{n(n+1)}{2}E_{n-1}(t) \vspace{3mm}\\
&=& \displaystyle E_{n+1}(t)- \binom{n+1}{2}t E_{n-1}(t).
\end{array}
$$
\end{proof}

\begin{corollary} \label{cor: poincare1}The Poincar\'e polynomial of $X(J)$ with $J=\{s_n\}$ is given by:
\begin{equation}
P(t)=E_{n+1}(t^2)-\binom{n+1}{2} t^2 E_{n-1}(t^2).
\end{equation}
\end{corollary}

\begin{proof}
We use the relation between the $h$-polynomial and the Poincar\'e polynomial recorded in \cite{Stan}, namely $$h(t^2)=P(t).$$ Hence our result follows immediately.
\end{proof}
\vspace{5mm}
     
Next, we generalize the previous result for  $1 \le k$ to the case of $$J(k,n)=\{s_{n-k+1}, s_{n-k+2}, \cdots ,s_n\} \subseteq S.$$ The main result of this section is a recurrence relation for the Poincar\'e polynomial of $X(J(k,n))$ in terms of the $(n-k)$-Eulerian polynomials.
\begin{theorem} \label{theo: eulerian4}${\rm Let} \ J(k,n)=\{s_{n-k+1}, s_{n-k+2}, \cdots ,s_n\} \ {\rm for} \ 1 \le k \le n$, and let $h_k(t)$ denote the $h$-polynomial of $X(J(k,n))$. Then $J(k,n)$ is combinatorially smooth. The following recurrence relation holds:
\begin{equation}
h_k(t)= h_{k-1}(t)- \binom{n+1}{k+1} (t^{k} +t^{k-1}+ \cdots +t)E_{n-k}(t).
\end{equation}
\end{theorem}

\begin{proof} From Corollary 3.5 \cite{R6} we obtain that $J(k,n)$ is combinatorially smooth. Thus $X(J(k,n))$ is rationally smooth as the corresponding polytopes $P_{\lambda,k}$ are simple polytopes. Let $M$ be a $\mathcal{J}$-irreducible monoid of type $J(k,n)$, where $$J(k,n)=\{s_{n-k+1}, s_{n-k+2}, \cdots ,s_n\} \subseteq S,$$ and let $\Lambda(k)$ denote the cross section lattice associated to $M$. From Theorem 2 we have that $\Lambda(k) \setminus \{0\} \cong S(J(k,n))$, where $S(J(k,n))$ was defined previously as $S(J(k,n))= \{I \subset S \ | \ {\rm no \ connected \  component \ of\ }I{\rm \ is \ contained \ entirely \  in\ } J(k,n) \}$. 
Next, consider $$M_i=\{A \subseteq S \ | \ J(k+1,n)-J(i,n) \subseteq A \subseteq S \setminus J(i,n)\},$$ for $0 \le i \le k+1$. 
So, in particular,
$$M_0=\{A \subseteq S \ | \ J(k+1,n) \subseteq A \},$$ and
$$M_{k+1}=\{A \subseteq S \ | \ A \subseteq S \setminus J(k+1,n)\}.$$
Hence, we have $$S(J(k,n))=\bigsqcup_{i=0}^{k+1} M_i,$$ for $0 \le i \le k+1$.

We associate to each $A \subseteq S(J(k,n))$, $$A^*_k= A \sqcup \{s \in J(k,n) \ | \ st=ts \ {\rm for \  any } \ t \in A \}.$$

We compute the $h$-polynomial of $X(J(k,n)$ using Corollary~\ref{cor: hpoly} and obtain:
\begin{equation}
h_{k}(t) = \sum_{i=0}^{k+1}\sum_{ A \subseteq M_i} \frac{(n+1)!}{|W_{A_{k}^{*}}|} (t-1)^{|A|}.
\end{equation}

\noindent Then for $A \subseteq M_0$  and for $A \subseteq M_1$, we have $A^{*}_k= A$ and $W_{A^{*}_k}=W_A$. For $A \subseteq M_i$, for $2 \le i \le k+1$ we have $A^{*}_k= A \sqcup J(i-1)$  and $W_{A^{*}_k}=W_A \times S_{i}.$
Thus, the $h$-polynomial of $X(J(k,n))$ is given by:
\begin{equation}
h_{k}(t) = \sum_{i=0}^{k+1} \sum_{ A \subseteq M_i} \frac{(n+1)!}{i!\times |W_A|} (t-1)^{|A|}.
\end{equation}
Consider $J(k-1, n)= \{s_{n-k+2}, \cdots ,s_n \}$. Then the cross-section lattice $\Lambda(k-1) \setminus \{0\}$ corresponding to $J(k-1, n)$ has the property that $\Lambda(k-1) \setminus \{0\} \cong S(J(k-1,n))$ according to Theorem 2.

Let $$S_i=\{A \subseteq S \ | \ J(k,n) \setminus J(i,n) \subseteq A \subseteq S \setminus J(i,n)\},$$ for $0 \le i \le k$. So, in particular, $$S_0=\{A \subseteq S \ | \ J(k,n) \subseteq A\},$$ and $$S_{k}=\{A \subseteq S \ | \ A \subseteq S \setminus J(k,n)\}.$$
Note that for each $i=0, \cdots ,k$, we have 
$$
S_i \cap \{A \subseteq S \ | \ s_{n-k} \in A\}=M_i,$$ and 
$$S_i \cap \{A \subseteq S \ | \ s_{n-k} \notin A\}=\{A \subseteq S  \ | \ J(k,n) \setminus J(i,n) \subseteq A \subset S \setminus (J(i,n) \cup \{s_{n-k}\}).$$
Let 
$$
\begin{array}{rcl}
N_i&=& \displaystyle \{A \subseteq S  \ | \ J(k,n) \setminus J(i,n) \subseteq A \subseteq S \setminus (J(i,n) \cup \{s_{n-k}\})\vspace{3mm} \\
&=& \displaystyle\{A' \cup (J(k,n) \setminus J(i,n)) \ | A' \subseteq S \setminus J(k+1,n)\}
\end{array}
$$ Note also that $N_k=M_{k+1}$. Then the following relation holds:
$$
S(J(k-1,n))= S(J(k,n)) \sqcup \bigsqcup_{i=0}^{k-1}N_i.
$$
Next we compute $A^*_{k-1}=A^*_{J(k-1,n)}$ for all $A \subseteq S(J(k-1,n)$. For $A \subseteq M_0$ and $A \subseteq M_1$, we have $$A^*_{k-1}=A \ {\rm and} \ W_{A^*_{k-1}}=W_A.$$ For $A \subseteq M_i: i=2 \cdots k$, we have $$A^*_{k-1}=A \cup J(k-1,n) \ {\rm and } \ W_{A^*}=W_A \times S_i.$$
We know that
$$
M_{k+1}=S \setminus J(k+1,n).
$$
Hence for $A \subseteq M_{k+1}$, we have $$A^*_{k-1}=A \cup J(k-1,n) \ {\rm and} \ W_{A^*}=W_A \times S_k.$$ For $A \subseteq N_0$ and $A \subseteq N_1$, we have $$A^*_{k-1}=A \ {\rm and} \ W_{A^*_{k-1}}=W_A.$$
For $A \subseteq N_i$ where $i =1, \cdots ,k-1$, we have $$A^*_{k-1}=A \cup J(i-1,n) \ {\rm and} \ W_{A^*_{k-1}}=W_A \times S_i.$$
Furthermore, the $h$-polynomial of $X(J(k-1,n))$ is given by:
\begin{equation}
\begin{array}{rcl}
h_{k-1}(t)&=& \displaystyle \sum_{i=0}^{k}\sum_{ A \subseteq M_i} \frac{(n+1)!}{i! \times |W_A|} (t-1)^{|A|}+ \vspace{3mm}\\
&&+\displaystyle \sum_{i=0}^{k-1}\sum_{ A \subseteq N_{i}} \frac{(n+1)!}{|W_A| \times i!} (t-1)^{|A|} \vspace{3mm}\\
&=& \displaystyle\sum_{A \subseteq M_{k+1}} \frac{|W|}{k! \times |W_A|}(t-1)^{|A|}.
\end{array}
\end{equation}
By (10) and (11), this implies that
\begin{equation}
\begin{array}{rcl}
h_{k-1}(t)-h_k(t)&=& \displaystyle\sum_{i=0}^{k-1}\sum_{ A \subseteq N_{i}} \frac{(n+1)!}{i! \times |W_A|} (t-1)^{|A|}+ \vspace{3mm}\\
&& +\displaystyle(\frac{1}{k!} - \frac{1}{(k+1)!}) \sum_{A \subseteq M_{k+1}} \frac{(n+1)!}{|W_A|}.
\end{array}
\end{equation}

The following relations hold: for $A \subseteq N_i$ there exists $A' \subseteq M_{k+1}$ such that $$A=A' \cup (J(k,n) \setminus J(i,n)) \ {\rm and} \ W_A=W_{A'} \times S_{k-i+1}.$$ 
We use these relations in (12) and obtain the following:
$$
\begin{array}{rcl}
h_{k-1}(t)-h_k(t)&=& \displaystyle\sum_{i=0}^{k-1} \sum_{A' \subseteq M_{k+1}}\frac{(n+1)!}{i! (k-i+1)!|W_{A'}|} (t-1)^{|A'|+k-i}\;\vspace{1mm}\\
&& +\displaystyle \frac{(n+1)!}{(n-k)!}(\frac{1}{k!}- \frac{1}{(k+1)!}) \sum_{A \subseteq M_{k+1}} \frac{(n-k)!}{|W_A|} (t-1)^{|A|}\;\vspace{3mm}\\
&=& \displaystyle\sum_{i=0}^{k-1}\frac{(n+1)!}{(n-k)!(k-i+1)! i!} (t-1)^{k-i}\sum_{A' \subseteq M_{k+1}} \frac{(n-k)!}{|W_{A'}|} (t-1)^{|A'|}\;\vspace{1mm}\\
&& \displaystyle +\frac{(n+1)!}{(n-k)!(k+1)!} k \sum_{A \subseteq M_{k+1}} \frac{(n-k)!}{|W_{A}|} (t-1)^{A|}.
\end{array}
$$
Theorem~\ref{theo: eulerian2} allows us to express the $(n-k)$-Eulerian polynomial in terms of the subsets of $M_{k+1}$:
$$
E_{n-k}(t)=\sum_{A \subseteq M_{k+1}}\frac{(n-k)!}{|W_A|} (t-1)^{|A|}.
$$
In the next formula we replace $(k-i+1)! \times i!$ by $\displaystyle \frac{1}{(k+1)!} \binom{k+1}{i}$ and obtain:
\begin{equation}
\begin{array}{rcl}
h_k(t)-h_{k-1}(t)&=& \displaystyle\sum_{i=0}^{k-1} \frac{(n+1)!}{(n-k)!(k-i+1)! \times i!} (t-1)^{k-i} E_{n-k}(t)+\;\vspace{3mm}\\
&& \displaystyle \frac{(n+1)!}{(n-k)!(k+1)!} k E_{n-k}(t) \vspace{3mm}\\
&=& \displaystyle\left[
\sum_{i=0}^{k-1} \frac{(n+1)!}{(n-k)!(k+1)!} \binom{k+1}{i} (t-1)^{k-i}\right.\;+\vspace{1mm}\\
&& \displaystyle \left.\frac{(n+1)!}{(n-k)!(k+1)!} k\right]
E_{n-k}(t)\;\vspace{3mm}\\
&=& \displaystyle\binom{n+1}{k+1}\left[\sum_{i=0}^{k-1} \binom{k+1}{i}(t-1)^{k-i}+k\right]E_{n-k}(t).
\end{array}
\end{equation}

We need now to show that \begin{equation}\sum_{i=0}^{k-1}(\binom{k+1}{i}(t-1)^{k-i}+k)=\sum_{i=1}^k t^i.
\end{equation}

Let $$f(t)=\sum_{i=0}^{k-1}\binom{k+1}{i}(t-1)^{k-i}.$$ Observe that
$$
\displaystyle \sum_{i=0}^{k+1} \binom{k+1}{i} (t-1)^{k+1-i}=(t-1)f(t)+\binom{k+1}{k}(t-1)+\binom{k+1}{k+1}(t-1)^0.$$

By the binomial theorem, the left hand side is $$t^{k+1}=((t-1)+1)^{k+1}.$$
So $$f(t)+k=\frac{t^{k+!}-(k+1)(t-1)-1}{t-1}+k=\frac{t^{k+1}-t}{t-1}=\sum_{i=1}^k t^i.$$
The theorem now follows from (13) and (14).
\end{proof}

\begin{corollary} \label{cor: poincare4} Let $P_k(t)$ be the Poincar\'e polynomial of $X(J(k,n))$. Then the following formula holds: 
$$P_1(t)-P_{k}(t)=\sum_{i=1}^{k} \binom{n+1}{i+1} (t^{2i}+ \cdots +t^2)E_{n-i}(t^2).$$
\end{corollary}
\begin{proof}
From the recurrence obtained in the previous theorem, we have that $$h_1(t)-h_k(t)=\sum_{i=1}^{k} \binom{n+1}{i+1}(t^{i}+ \cdots +t)E_{n-i}.$$
\end{proof}

\begin{example}
Next, we verify the recurrence formula obtained in Theorem 6.

Let $k=n-1$. Then $J(n-1,n)=\{s_2.s_3. \cdots ,s_n\}$ is combinatorially smooth and the $h$-polynomial of $X(J(n-1,n))$ computed in Example 4.3 \cite{R6}, is given by the formula:
$$
h_{n-1}(t)=1+t+t^2+ \cdots +t^{n}.
$$
The recurrence formula obtained in Theorem 6 is equivalent to
$$
h_{n-1}=h_{n-2}(t)- \binom{n+1}{n}(t^{n-1}+ \cdots +t)E_1(t).
$$
We know the $h$-polynomial of $X(J(n-2,n))$ from Example 4.6 \cite{R7}. It is given by the following formula:
$$
h_{n-2}(t)=1+(n+2)t+(n+2)t^2+..+(n+2)t^{n-1} + t^{n}.
$$
Hence $$h_{n-2}(t)-h_{n-1}(t)=n(t^{n-1}+ \cdots +t)$$ yields the desired relation.
\end{example}

{\bf Acknowledgements.} I would like to thank my advisor, Lex Renner, for his supervision during my graduate studies and for bringing to my attention this problem, and Nicole Lemire for helpful conversations. I also thank Oleg Golubitsky, for helping me with \LaTeX.

\end{document}